\documentclass[reqno]{amsart}
\usepackage{amsmath}
\usepackage{amssymb}
\usepackage{amsthm}
\usepackage{mathrsfs}
\usepackage{url}
\newtheorem{definition}{Definition}[section]
\newtheorem{theorem}[definition]{Theorem}

\newtheorem{prop}[definition]{Proposition}
\newtheorem{cor}[definition]{Corollary}

\newtheorem{examples}[definition]{Examples}
\newtheorem{remark}[definition]{Remark}

\usepackage{color}
\numberwithin{equation}{section}

\newcommand{\C}{\mathbb{C}}
\newcommand{\D}{\mathbb{D}}
\newcommand{\B}{B_{\delta}}
\newcommand{\E}{\mathcal{E}}
\newcommand{\dd}{d_{\mathbb{D}}}
\begin{document}

\title[Schwarz-Pick and von Neumann's inequalities]{Some relations between Schwarz-Pick inequality and von Neumann's inequality}
\author{Kenta Kojin }
\address{
Graduate School of Mathematics, Nagoya University, 
Furocho, Chikusaku, Nagoya, 464-8602, Japan
}
\email{m20016y@math.nagoya-u.ac.jp}
\date{\today}
\thanks{This work was supported by JSPS Research Fellowship for Young Scientists (KAKENHI Grant Number JP 23KJ1070)} 

\begin{abstract}
We study a Schwarz-Pick type inequality for the Schur-Agler class $SA(B_{\delta})$. In our operator theoretical approach, von Neumann's inequality for a class of generic tuples of $2\times 2$ matrices plays an important role rather than holomorphy. In fact, the class $S_{2, gen}(B_{\Delta})$ consisting of functions that satisfy the inequality for those matrices enjoys
\begin{equation*}
\dd(f(z), f(w))\le d_{\Delta}(z, w)
\;\;(z,w\in B_{\Delta}, f\in S_{2, gen}(B_{\Delta})).
\end{equation*}
Here, $d_{\Delta}$ is a function defined by a matrix $\Delta$ of functions. Later, we focus on the case when $\Delta$ is a matrix of holomorphic functions. We use the pseudo-distance $d_{\Delta}$ to give a sufficient condition on a diagonalizable commuting tuple $T$ acting on $\C^2$ for $B_{\Delta}$ to be a complete spectral domain for $T$. We apply this sufficient condition to generalizing von Neumann's inequalities studied by Drury \cite{Dru1983} and by Hartz-Richter-Shalit \cite{HRS}.

\end{abstract}

 \maketitle


\section{Introduction}
Every holomorphic function $f$ from the open unit disk $\D$ into itself satisfies the Schwarz-Pick inequality
\begin{equation*}
\left|\frac{f(z)-f(w)}{1-\overline{f(w)}f(z)}\right|\le\left|\frac{z-w}{1-\overline{w}z}\right|
\end{equation*}
for any pair $z,w\in\D$ \cite[page 5]{Din}. The pseudo-hyperbolic distance $\dd$ on $\D$ is defined by the right-hand side in the above inequality. 
One of the natural generalizations of $\dd$ on a domain $\Omega\subset\C^d$ is the M$\ddot{\mathrm{o}}$bius pseudo-distance $d_{\Omega}$. We remark that a Schwarz-Pick type inequality holds with respect to this pseudo-distance (see section 2). However, the inequality may not be the best estimate for the Schur-Agler class on a polynomial polyhedron $B_{\delta}$, denoted by $SA(B_{\delta})$.
Here, $\delta$ is a matrix of polynomials in $d$-variables and 
\begin{equation*}
B_{\delta}:=\{z\in\C^d\;|\;\|\delta(z)\|<1\}.
\end{equation*}
The Schur-Agler class is a generalization of the Schur class of $\D$ that plays an important role in operator theory and complex analysis (see e.g., \cite{Agl, AMg, AM2015oka, AMY, AT, BH2013, BMV}). Moreover, it is difficult to determine the explicit formula of the M$\ddot{\mathrm{o}}$bius pseudo-distance $d_{\Omega}$. In section 4, we will define a pseudo-distance $d_{\delta}$ on $B_{\delta}$ by
\begin{align*}
d_{\delta}(z,w):=&\left\|(I-\delta(w)\delta(w)^*)^{-\frac{1}{2}}(\delta(z)-\delta(w))\right.\\
&\;\;\;\;\;\;\left.\times(I-\delta(w)^*\delta(z))^{-1}(I-\delta(w)^*\delta(w))^{\frac{1}{2}} \right\|.
\end{align*}
and prove that 
\begin{equation*}
\dd(f(z), f(w))\le d_{\delta}(z,w)\le d_{B_{\delta}}(z,w).
\end{equation*}
holds for any $f\in SA(B_{\delta})$ and $z,w\in B_{\delta}$.
In section 5, this estimate will be shown to be the best for the Schur-Agler class $SA(B_{\delta})$ (see Proposition \ref{prop1}). We will prove this inequality by operator theoretical approach. We emphasize that von Neumann's inequality for a class of generic tuples of $2\times 2$ matrices plays an important role rather than holomorphy. This observation leads us to define a class $S_{2, gen}(B_{\Delta})$ of functions and a pseudo-distance $d_{\Delta}$ on $B_{\Delta}$ by
\begin{align*}
d_{\Delta}(z,w):=&\left\|(I-\Delta(w)\Delta(w)^*)^{-\frac{1}{2}}(\Delta(z)-\Delta(w))\right.\\
&\;\;\;\;\;\;\left.\times(I-\Delta(w)^*\Delta(z))^{-1}(I-\Delta(w)^*\Delta(w))^{\frac{1}{2}} \right\|
\end{align*}
in sections 3 and 4. Here, $\Delta$ is a matrix of functions on a set $E\subset\C^d$ and
\begin{equation*}
B_{\Delta}:=\{z\in E\;|\;\|\Delta(z)\|<1\}.
\end{equation*}
We will also show 
\begin{equation*}
\dd(f(z), f(w))\le d_{\Delta}(z,w)
\end{equation*}
for any $f\in S_{2, gen}(B_{\Delta})$ and $z,w\in B_{\Delta}$ (see Corollary \ref{cor1}). 

In section 5, we will focus on the case when $\Delta$ is a matrix of holomorphic functions and prove the following result (see Theorem \ref{theorem2}):
\begin{theorem}
Let $T$ be a diagonalizable commuting $d$-tuple of operators acting on $\C^2$ with $\sigma(T)=\{z,w\}\subset B_{\Delta}$. If $\|\Delta(T)\|\le 1$ and $d_{\Delta}(z,w)=d_{B_{\Delta}}(z,w)$, then $B_{\Delta}$ is a complete spectral domain for $T$. 
\end{theorem}

As an application of this theorem, we can prove the following relation between a certain Schwarz-Pick type inequality and von Neumann's inequality (see Corollary \ref{cor3}):

\begin{theorem}
Let $B_{\Delta}$ be a domain in $\C^d$ associated with a matrix of holomorphic functions $\Delta$. Then the following conditions are equivalent:
\begin{enumerate}
\item $d_{\Delta}=d_{B_{\Delta}}$.
\item If a diagonalizable $d$-tuple $T$ of operators acting on $\C^2$ with $\sigma(T)\subset B_{\Delta}$ satisfies $\|\Delta(T)\|\le1$, then $B_{\Delta}$ is a complete spectral domain for $T$.
\end{enumerate}
\end{theorem}

We apply this result to proving von Neumann's inequality for pairwise commuting tuples of $2\times 2$ row contractions. 
This result is a generalization of von Neumann's inequalities studied by Drury \cite{Dru1983} and by Hartz-Richter-Shalit \cite{HRS}.


\section{Preliminaries}

We review some notions and basic results on complex geometry and in operator theory. The set of holomorphic mappings from $\Omega_1\subset\C^d$ to $\Omega_2\subset \C^{d'}$ is denoted by $H(\Omega_1, \Omega_2)$. 
Let $\rho_{\D}$ denote the Poincar\'{e} distance on the open unit disk $\D$. Namely, 
\begin{equation*}
\rho_{\D}(z,w)=\tanh^{-1}\left|\frac{z-w}{1-\overline{w}z}\right|\;\;\;\;(z,w\in\D).
\end{equation*}
For a domain $\Omega\subset\C^d$, the Carath\'{e}odory pseudo-distance is defined by
\begin{equation*}
C_{\Omega}(z,w)=\sup_{f\in H(\Omega, \mathbb{D})}{\rho_{\D}(f(z), f(w))}\;\;(z,w\in\Omega).
\end{equation*}

In particular, $\rho_{\D}(f(z),f(w))\le C_{\Omega}(z,w)$ holds for every $f\in H(\Omega, \D)$. Moreover, 
it is well known that the Carath\'{e}odory pseudo-distances form the smallest Schwarz-Pick system (and the Kobayashi pseudo-distances form the largest one); see \cite[Proposition 3.1.7]{Kob} .

For simplicity, we will consider the following reformulation of the Carath\'{e}odory pseudo-distance. Recall that the pseudo-hyperbolic distance on $\D$ is defined by
\begin{equation*}
d_{\D}(z,w)=\left|\frac{z-w}{1-\overline{w}{z}}\right|\;\;(z, w\in \D).
\end{equation*}
Moreover, we define $d_{\D}(\alpha, \alpha)=0$ if $|\alpha|=1$. We will use this definition because all the main classes of functions in this paper contain the constant functions with modulus one. Since their behaviors are trivial, they are not essential. 
If one wants to avoid the use of this definition, it is reasonable to exclude the constant functions with modulus one from all classes. For a domain $\Omega$, the M$\ddot{\mathrm{o}}$bius pseudo-distance $d_{\Omega}$ is defined by
\begin{equation*}
d_{\Omega}(z,w)=\sup_{f\in H(\Omega, \mathbb{D})}{\dd(f(z), f(w))}\;\;(z,w\in\Omega).
\end{equation*}
We note that the relation between $d_{\Omega}$ and $C_{\Omega}$ is given by
\begin{equation*}
C_{\Omega}=\tanh^{-1}(d_{\Omega}).
\end{equation*}
Since $\tanh^{-1}(x)$ is strictly increasing on the interval $(-1, 1)$, it is sufficient to consider $d_{\Omega}$ for studying Schwarz-Pick type results.

\begin{examples}\label{examples:2.2}
\upshape
\begin{enumerate}
\item For $\Omega = \D$, the celebrated Schwarz-Pick lemma implies that the M$\ddot{\mathrm{o}}$bius pseudo-distance coincides with the pseudo-hyperbolic distance on $\D$.
\item For any domains $\Omega_1$ and $\Omega_2$, we have 
\begin{equation*}
d_{\Omega_1\times\Omega_2}((z^1, z^2), (w^1, w^2))=\max\{{d_{\Omega_1}(z^1, w^1), d_{\Omega_2}(z^2, w^2)}\}
\end{equation*}
(see \cite[Theorem 4.9.1]{Kob}). In particular,
\begin{equation*}
d_{\D^d}((z^1,\ldots, z^d), (w^1,\ldots,w^d))=\max_{1\le i\le d}\{\dd(z^i, w^i)\}.
\end{equation*}
\end{enumerate}
\end{examples}


Next, we will define the Schur-Agler class. Let $\delta$ be an $s\times r$ matrix of polynomials in $d$-variables. Let
\begin{equation*}
B_{\delta}=\{z\in\C^d\;|\;\|\delta(z)\|<1\}.
\end{equation*}
Here $\|\delta(z)\|$ denotes the operator norm. We fix a separable Hilbert space $H$ and let $CB(H)^d$ be the set of commuting $d$-tuples of bounded operators $T=(T^1,\ldots,T^d)$ on $H$. Ambrozie and Timotin \cite[Lemma 1]{AT} showed that if $T\in CB(H)^d$ satisfies $\|\delta(T)\|<1$, then the Taylor spectrum $\sigma(T)$ is contained in $B_{\delta}$. Therefore, we can consider the norm $\|f(T)\|$ for such $T$ with any $f\in H(B_{\delta},\C)$ via the Taylor functional calculus \cite{Cur, Tay70}. The Schur-Agler class on $B_{\delta}$, $SA(B_{\delta})$, is constructed by holomorphic functions that satisfy a certain von Neumann inequality:
\begin{equation*}
SA(B_{\delta}):=\left\{f\in H(B_{\delta}, \C)\;\middle|\;\|f(T)\|\le1\;\mbox{whenever}\;\|\delta(T)\|<1\right\}.
\end{equation*}

We note that $SA(B_{\delta})$ is independent of the choice of separable Hilbert space $H$ because the Taylor functional calculus respects intertwinings \cite[Theorem 4.5]{Tay70}. It is obvious that $SA(B_{\delta})\subset H(B_{\delta},\overline{\D})$. 

The following result by Ambrozie and Timotin \cite[Theorem 3]{AT} suggested by Agler's observations \cite{Agl} will play an important role in this paper. 


\begin{theorem}$($\cite[Theorem 3]{AT}$)$\label{Ambrozie-Timotin}
Let $S\subset B_{\delta}$ and $\phi:S\rightarrow\C$ be a function. Then the following conditions are equivalent:
\begin{enumerate}
\item There exists an $f\in SA(\B)$ such that $f|_{S}=\phi$.
\item There exist an auxiliary Hilbert space $\E$ and a unitary operator
\begin{equation*}
U=
\begin{bmatrix}
A & B\\
C & D
\end{bmatrix}:
\begin{bmatrix}
\E\otimes\C^s\\
\C
\end{bmatrix}\rightarrow
\begin{bmatrix}
\E\otimes\C^r\\
\C
\end{bmatrix}
\end{equation*}
such that $\phi$ has a transfer-function realization
\begin{equation*}
\phi(z)=D+C(I_{\E\otimes\C^s}-(I_\E\otimes\delta(z))A)^{-1}(I_\E\otimes\delta(z))B
\end{equation*}
for all $z\in S$.
\end{enumerate}
\end{theorem}


In the test-function approach, Ball and Huam$\acute{\mathrm{a}}$n \cite{BH2013} defined generalized Schur-Agler classes and we will use their characterizations in section 4. Here, we recall their definitions only in the scalar-valued case. Let $\mathcal{U}_T$ and $\mathcal{Y}_T$ be two coefficient Hilbert spaces and let $\Psi$ be a collection of functions $\psi$ on the abstract set of points $E$ with values in the space $B(\mathcal{U}_T, \mathcal{Y}_T)$. We say that $\Psi$ is a {\bf collection of test functions} if
\begin{equation*}
\sup\{\|\psi(z)\|\;|\;\psi\in\Psi\} < 1\;\;\text{for each $z\in E$}.
\end{equation*}
For a positive semidefinite function $k:E\times E\rightarrow \C$ , we say that $k$ is $\Psi$-{\bf admissible} if the $\mathbb{C}$-valued function 
\begin{equation*}
\tilde{k}_{X,\psi}(z,w) :=X(w)^*(I-\psi(w)^*\psi(z))X(z)k(z,w)
\end{equation*}
is positive semidefinite on $E\times E$ for each choice of $X:E\rightarrow\mathcal{C}_2(\C, \mathcal{U}_T)$ and $\psi\in\Psi$.
Here, $\mathcal{C}_2(\C, \mathcal{U}_T)$ denotes the space of Hilbert-Schmidt class operator from $\C$ into $\mathcal{U}_T$.
We then say that a function $f:E\rightarrow \C$ is in the $\Psi$-Schur-Agler class $SA_{\Psi}(\C)$ if the $\C$-valued function
\begin{equation*}
\check{k}_{f}(z,w):=(1-f(w)^*f(z))k(z,w)
\end{equation*}
is positive semidefinite on $E\times E$ for each $\Psi$-admissible $k$.


\section{$S_{2, gen}(B_{\Delta}$) and $S_2(B_{\Delta})$}


If $T=(T^1,\ldots,T^d)$ is a commuting $d$-tuple of operators acting on $\C^2$, we say that $T$ is {\bf generic} if there exist 2 linearly independent joint eigenvectors whose corresponding joint eigenvalues are distinct (or equivalently, $T$ is diagonalizable and $\sigma(T)\subset\C^d$ has cardinality 2). Thus, if $T$ is generic with $\sigma(T)=\{z_1, z_2\}$, then there exist 2 linearly independent vectors $v_1$, $v_2\in\C^2$ such that
\begin{equation}\label{def1}
T^rv_j = z_j^rv_j
\end{equation}
with $z_j=(z_j^r)_{1\le r\le d}$ for $j=1, 2$. So, if $f$ is a function on $\{z_1, z_2\}$, then we define $f(T)$ to be a unique operator on $\C^2$ that satisfies
\begin{equation}\label{def2}
f(T)v_j = f(z_j)v_j=f(z_j^1,\ldots,z_j^d)v_j\;\;(j=1, 2).
\end{equation}

Moreover, if $\Delta$ is an $s\times r$ matrix of functions on $\{z_1, z_2\}$, then we define $\Delta(T)\in B(\C^r\otimes\C^2, \C^s\otimes\C^2)$ by
\begin{equation*}
\Delta(T): e_i\otimes v_j\mapsto \Delta(z_j)e_i\otimes v_j
\end{equation*}
for $i = 1,\ldots, r$ and $j = 1,2$, where $\{e_i\}_{i=1}^r$ is the standard basis for $\C^r$.

Let $E$ be a subset in $\C^d$ and let $\Delta$ be an $s\times r$ matrix of functions on $E$. We define
\begin{equation*}
B_{\Delta}:=\{z\in E\;|\;\|\Delta(z)\|<1\}
\end{equation*}
and we always assume that $B_{\Delta}$ is larger than a singleton.
We define a class $S_{2,gen}(B_{\Delta})$ to be the set of functions $f$ on $B_{\Delta}$ such that the norm
\begin{equation*}
\|f\|_{2, gen}:=\sup \{\|f(T)\|\;|\;\text{$T$ is generic with $\sigma(T)\subset B_{\Delta}$ and $\|\Delta(T)\|\le1$}\}.
\end{equation*}
is less than or equal to 1, where the subscript ``2" emphasizes that we are only using $2\times2$ matrices.
Note that $S_{2, gen}(B_{\Delta})$ is a closed set in the pointwise convergence topology. If $E$ is an open set and the components of a matrix $\Delta$ are holomorphic functions, we define 
\begin{equation*}
S_2(B_{\Delta}):=S_{2,gen}(B_{\Delta})\cap H(B_{\Delta}, \C)\subset H(B_{\Delta}, \overline{\D}). 
\end{equation*}
This is the Schur class of the set of generic operators with $\sigma(T)\subset B_{\Delta}$ and $\|\Delta(T)\|\le 1$ \cite{AMY2019} (see also \cite[Section 9]{AMY}). In section 5, we will give a characterization of $\Delta$ such that $S_2(B_{\Delta})$ and $H(B_{\Delta},\overline{\D})$ coincide.

\begin{examples}\upshape


(1) In section 4, we will give a characterization functions in $S_{2, gen}(B_\Delta)$ and use it to see $SA(B_{\delta})\subset S_2(B_{\delta})$ in conjunction with Theorem \ref{Ambrozie-Timotin}. In particular, the set of contractive multipliers of the Drury-Arveson space \cite{AMp, Hartz2022} and the Bedea-Beckermann-Crouzeix class \cite[Section 9.6]{AMY} are examples of such functions.

(2) For any diagonal matrix $\Delta$ of holomorphic functions, the closed unit ball of $H_{\Delta, \text{gen}}^{\infty}(B_{\Delta})$ is a subclass of $S_{2, gen}(B_{\Delta})$, where the definition of the closed unit ball of $H_{\Delta, gen}^{\infty}(B_{\Delta})$ is the same as that of $S_{2, gen}(B_{\Delta})$ with replacing $T$ acting on $\C^2$ with $T$ acting on $\C^m$ for arbitrary $m\in\mathbb{N}$. This class was studied in \cite{AM2015oka, AMY2019} and the definition of $S_{2,gen}(B_{\Delta})$ is based on it.

(3) If $d\ge 3$ and $\delta_d(z)=\mathrm{diag}(z^1,\ldots, z^d)$, it is known that $B_{\delta_d}=\D^d$ and $SA(B_{\delta_d})\subsetneq H(\D^d, \overline{\D})$ (see \cite{Agl} or \cite{AT}). However,  $S_2(B_{\delta_d})=H(\D^d, \overline{\D})$ since von Neumann's inequality holds for any $d$-tuple of commuting $2\times 2$ matrices \cite{Dru1983}.
\end{examples}




\begin{remark}\label{remark}

\upshape (1) 
In the simplest case that $E=\C$ and $\delta_1=[z]$, i.e., $B_{\delta_1}=\D$, we observe that $S_2(B_{\delta_1})\neq S_{2, gen}(B_{\delta_1})$. In fact, we can see that the function $\overline{\chi}(z):=\overline{z}$ is in $S_{2, gen}(B_{\delta_1})$. More generally,  we can prove that $\|f\|_{2, gen}\le 1$ if and only if $\displaystyle\sup_{z\ne w}\frac{\dd(f(z), f(w))}{\dd(z,w)}\le 1$. Therefore, $S_{2, gen}(B_{\delta_1})$ is a set of Lipschitz continuous-like functions rather than holomorphic functions.


(2) 
Let us consider the classes  $S_{3, gen}(B_{\Delta})$ and $S_3(B_{\Delta})$. The definitions of these classes are the same as those of $S_{2, gen}(B_{\Delta})$ and $S_2(B_{\Delta})$ with replacing $T$ acting on $\C^2$ with $T$ acting on $\C^3$. If $\Delta$ is a diagonal matrix of holomorphic functions or a row matrix of holomorphic functions, then $S_{3, gen}(B_{\Delta})=S_3(B_{\Delta})$ holds. In fact, by the same argument as in the proof of Theorem \ref{Theorem1} (see section 4), \cite[Theorem 3,1(1)$\Rightarrow$(2)]{BH2013} and \cite[Proposition 2.3]{AHMS}, we get analyticity in $S_{3, gen}(B_{\Delta})$. Moreover, if $\delta_d(z)=\mathrm{diag}(z^1, \ldots, z^d)$, then von Neumann's inequality for $d$-tuples of commuting $3\times 3$ contractive matrices \cite{Kne} yields that $S_{3, gen}(B_{\delta_d})=S_3(B_{\delta_d})=H(\D^d, \overline{\D})$.

\end{remark}



\section{Schwarz-Pick type result for $S_{2, gen}(B_{\Delta})$}

In this section, we generalize \cite[Proposition 2.6]{AM2015oka} by transfer-function realizations.
We will first obtain a more concrete description of $S_{2, gen}(B_{\Delta})$. More precisely, we will show that every function in $S_{2, gen}(B_{\Delta})$ coincides with a $\Delta$-Schur-Agler function over any 2 points. 
As an application of this fact, for any two points in $B_{\Delta}$, every function in $S_{2,gen}(B_{\Delta})$ is subordinate to some Schur-Agler function on the following classical Cartan domain of type I:
\begin{equation*}
R_{sr}:=\left\{(\zeta^{1, 1},\ldots,\zeta^{s,r})\in\C^{sr}\;\middle|\left\|
\begin{bmatrix}
\zeta^{1,1}&\cdots&\zeta^{1,r}\\
\zeta^{2, 1}&\cdots&\zeta^{2,r}\\
\vdots&\ddots&\vdots\\
\zeta^{s,1}&\cdots&\zeta^{s,r}
\end{bmatrix}\right\|<1\right\}.
\end{equation*}
In this paper, we sometimes identify $(\zeta^{1, 1},\ldots,\zeta^{sr})\in R_{s,r}$ with the contractive $s\times r$ matrix $\begin{bmatrix}
\zeta^{1,1}&\cdots&\zeta^{1,r}\\
\zeta^{2, 1}&\cdots&\zeta^{2,r}\\
\vdots&\ddots&\vdots\\
\zeta^{s,1}&\cdots&\zeta^{s,r}
\end{bmatrix}$.
For any $\Lambda=\{z_1, z_2\}\subset B_{\Delta}$, we denote the $\Psi$-Schur-Agler class associated with the test function $\Delta|_{\Lambda}$ (so $\mathcal{U}_T=\C^r$, 
$\mathcal{Y}_T=\C^s$ and 
$\Psi=\{\Delta|_{\Lambda}\}$) by 
$SA_{\Delta|_{\Lambda}}$.


\begin{theorem}\label{Theorem1}
Let $f:B_{\Delta}\rightarrow\C$ be a function. Then the following conditions are equivalent:
\begin{enumerate}
\item $f\in S_{2, gen}(B_{\Delta})$.
\item $f\in SA_{\Delta|_{\Lambda}}$ with any $\Lambda=\{z_1, z_2\}\subset B_{\Delta}$.
\item For each $\Lambda=\{z_1, z_2\}\subset B_{\Delta}$,  there exist an auxiliary Hilbert space $\E$ and a unitary operator
\begin{equation*}
U=
\begin{bmatrix}
A & B\\
C & D
\end{bmatrix}:
\begin{bmatrix}
\E\otimes\C^s\\
\C
\end{bmatrix}\rightarrow
\begin{bmatrix}
\E\otimes\C^r\\
\C
\end{bmatrix}
\end{equation*}
such that $f$ has a transfer-function realization
\begin{equation*}\label{realization}
f(z_i)=D+C(I_{\E\otimes\C^s}-(I_\E\otimes\Delta(z_i))A)^{-1}(I_\E\otimes\Delta(z_i))B
\end{equation*}
for $i=1, 2$.
\item For each $\Lambda=\{z_1, z_2\}\subset B_{\Delta}$, 
there exists a Schur-Agler function $F\in SA(R_{sr})$ such that 
\begin{equation}\label{subord}
f(z_i)=F(\Delta(z_i))
\end{equation}
for $i=1,2$.
\end{enumerate}
\end{theorem}



\begin{proof}(2)$\Leftrightarrow$(3) follows from \cite[Theorem 3.1]{BH2013} (cf. \cite[Theorem 2.18, Theorem 2,20 and Remark 2.21]{BMV}).

Since $R_{sr}$ is the polynomial polyhedron associated with $\delta(z)=[\zeta^{i,j}]_{1\le i\le s,  1\le j \le r}$, Theorem \ref{Ambrozie-Timotin} implies (3)$\Leftrightarrow$(4).


(1)$\Rightarrow$(2). First, let $k:\Lambda\times\Lambda\rightarrow\C$ be a $\Delta|_{\Lambda}$-admissible positive
 {\bf definite} function (so it defines an invertible matrix $[k(z_i, z_j)]_{i,j=1}^2$). Thus, there exist 2 linearly independent vectors $v_1$ and $v_2\in\C^2$ such that $k(z_i,z_j)=\langle v_i, v_j\rangle_{\C^2}$ ($i=1, 2$). Since $\{v_1, v_2\}$ is a basis for $\C^2$, we can define the pairwise commuting operators $T=(T^1,\ldots,T^d)$ acting on $\C^2$ with $\sigma(T)=\Lambda$
 by equation (\ref{def1}). If $\xi$ is a vector in $\C^r$, we can view $\xi$ as an operator from $\C$ to $\C^r$. The adjoint operator $\xi^*:\C^r\rightarrow\C$ is then given by $\xi^*: \eta\mapsto\langle \eta, \xi\rangle_{\C^r}$. 
 For any pair $\xi_1, \xi_2\in\C^r$, we define $X:\Lambda\rightarrow\mathcal {C}_2(\C,\C^r)$ by $X(z_i):=\xi_i$ ($i=1, 2$). Then, we have
\begin{align}\label{equation1}
\tilde{k}_{X, \Delta |_{\Lambda}}(z_i,z_j)&=X(z_j)^*(I-\Delta(z_j)^*\Delta(z_i))X(z_i)k(z_i,z_j)\notag\\
&=\langle\xi_i\otimes v_i, \xi_j\otimes v_j\rangle-\langle(\Delta(z_i)\xi_i)\otimes v_i, (\Delta(z_j)\xi_j)\otimes v_j\rangle\notag\\
&=\langle\xi_i\otimes v_i, \xi_j\otimes v_j\rangle-\langle\Delta(T)(\xi_i\otimes v_i), \Delta(T)(\xi_j\otimes v_j)\rangle\notag\\
&=\langle (I-\Delta(T)^*\Delta(T))\xi_i\otimes v_j, \xi_j\otimes v_j\rangle.
\end{align}
Since $k$ is $\Delta|_{\Lambda}$-admissible, $\tilde{k}_{X, \Delta|_{\Lambda}}$ is positive semidefinite. So, we conclude that $I-\Delta(T)^*\Delta(T)\ge 0$. Therefore, we have $\|\Delta(T)\|\le 1$, and the definition of $S_{2,gen}(B_{\Delta})$ implies $I-f(T)^*f(T)\ge 0$. This is equivalent to that the function $\check{k}_f(z_i,z_j)=(1-f(z_j)^*f(z_i))k(z_i,z_j)$ is positive semidefinite. 

Next, let $k$ be a general $\Delta|_{\Lambda}$-admissible positive semidefinite function. For any $\varepsilon>0$, $\varepsilon\delta_{i,j}+k(z_i,z_j)$ is a $\Delta|_{\Lambda}$-admissible positive definite function, where $\delta_{i,j}$ is the Kronecker delta. We have already proved that $\check{k}_{\varepsilon, f}(z_i, z_j):=(1-f(z_j)^*f(z_i))(\varepsilon\delta_{i,j}+k(z_i, z_j))$ is positive semidefinite. Thus, we may take the limit as $\varepsilon\rightarrow 0$ to obtain that $\check{k}_f(z_i, z_j)=(1-f(z_j)^*f(z_i))k(z_i, z_j)$ is a positive semidefinite function. Therefore, $f\in SA_{\Delta|_{\Lambda}}$ holds for any $\Lambda=\{z_1, z_2\}\subset B_{\Delta}$.


(2)$\Rightarrow$(1). Let $T$ be a generic tuple with the Taylor joint spectrum $\Lambda:=\sigma(T)=\{z_1, z_2\}\subset B_{\Delta}$ and $\|\Delta(T)\|\le 1$. Let $v_1$, $v_2\in\C^2$ be corresponding joint eigenvectors. Set $k(z_i, z_j):=\langle v_i, v_j\rangle_{\C^2}$. Then, equation (\ref{equation1}) implies that $k$ is a $\Delta|_{\Lambda}$-admissible positive semidefinite function. Since  
$f\in SA_{\Delta|_{\Lambda}}$, 
$\check{k}_f(z_i, z_j)=(1-f(z_j)^*f(z_i))k(z_i, z_j)$ is a positive semidefinite function and this is equivalent to $I-f(T)^*f(T)\ge 0$. 
Thus, we conclude that $f\in S_{2, gen}(B_{\Delta})$.
 \end{proof}





As a corollary of this theorem, we obtain a Schwarz-Pick type inequality for $S_{2, gen}(B_{\Delta})$.
We define a function $d_{\Delta}:B_{\Delta}\times B_{\Delta}\rightarrow\mathbb{R}_{\ge 0}$ by
\begin{align*}
d_{\Delta}(z,w):=&\left\|(I-\Delta(w)\Delta(w)^*)^{-\frac{1}{2}}(\Delta(z)-\Delta(w))\right.\\
&\;\;\;\;\;\;\left.\times(I-\Delta(w)^*\Delta(z))^{-1}(I-\Delta(w)^*\Delta(w))^{\frac{1}{2}} \right\|_{B(\C^r, \C^s)}.
\end{align*}
We point out that Harris \cite{Har} characterized the automorphisms of $R_{sr}$, and it is easy to see that 
\begin{equation*}
d_{\Delta}(z,w)=d_{R_{sr}}(\Delta(z),\Delta(w)).
\end{equation*}
In fact, for any $\zeta=(\zeta^{1, 1},\ldots,\zeta^{s,r})$ and $\eta=(\eta^{1, 1},\ldots,\eta^{s,r})$ in $R_{sr}$, identified with $Z=\begin{bmatrix}
\zeta^{1,1}&\cdots&\zeta^{1,r}\\
\zeta^{2, 1}&\cdots&\zeta^{2,r}\\
\vdots&\ddots&\vdots\\
\zeta^{s,1}&\cdots&\zeta^{s,r}
\end{bmatrix}$ and 
$H=\begin{bmatrix}
\eta^{1,1}&\cdots&\eta^{1,r}\\
\eta^{2, 1}&\cdots&\eta^{2,r}\\
\vdots&\ddots&\vdots\\
\eta^{s,1}&\cdots&\eta^{s,r}
\end{bmatrix}$ respectively, we have
\begin{equation*}
d_{R_{sr}}(\zeta,\eta):=\left\|(I-HH^*)^{-\frac{1}{2}}(Z-H)(I-H^*Z)^{-1}(I-H^*H)^{\frac{1}{2}} \right\|_{B(\C^r, \C^s)}.
\end{equation*}


\begin{cor}\label{cor1}
If $f\in S_{2, gen}(B_{\Delta})$ and $z$, $w\in B_{\Delta}$, then
\begin{equation}\label{eq:Schwarz-Pick}
\dd(f(z), f(w))\le d_{\Delta}(z, w).
\end{equation}
\end{cor}

\begin{proof}
Since $f\in S_{2, gen}(B_{\Delta})$, Theorem \ref{Theorem1}(1)$\Rightarrow$(4) implies that there exists an $F\in SA(R_{sr})$ such that 
$f(z)=F(\Delta(z))$ and $f(w)=F(\Delta(w))$. If $|F(\Delta(z))|=1$ or $|F(\Delta(w))|=1$, then we apply the maximum modulus principle to getting (\ref{eq:Schwarz-Pick}). Otherwise, we have
\begin{equation*}
\dd(f(z), f(w))=\dd(F(\Delta(z)), F(\Delta(w)))\le d_{R_{sr}}(\Delta(z), \Delta(w))=d_{\Delta}(z, w).
\end{equation*}
Here, the above inequality follows from the definition of the M$\ddot{\mathrm{o}}$bius pseudo-distance $d_{R_{sr}}$.
\end{proof}



\begin{remark}
\upshape Theorem \ref{Theorem1} is a ``two points analogue" of \cite[Theorem 3.11]{AM2015oka}. Agler and McCarthy \cite[Proposition 2.6]{AM2015oka} proved their Schwarz-Pick type inequality in the case when $\Delta$ is a diagonal matrix of holomorphic functictions and use it to prove \cite[Theorem 3.11]{AM2015oka}. Although, we do not know how to apply their proof of \cite[Proposition 2.6]{AM2015oka} to our general setting, we will use their ideas to obtain some dilation results in the next section.
\end{remark}




\begin{examples}\upshape

(1) (Schur-Agler class) If $f\in SA(B_{\delta})$, then 
\begin{equation*}
\dd(f(z), f(w))\le d_{\delta}(z,w) \;\;(z,w\in B_{\delta}).
\end{equation*}
In the next section, we will see that this inequality is the best possible for $SA(B_{\delta})$.

(2) (Schur class of $\D^d$) If $f\in H(\D^d,\overline{\D})=S_2(B_{\delta_d})$, then 
\begin{equation*}
\dd(f(z), f(w))\le \|diag(\dd(z^1, w^1),\ldots, \dd(z^d, w^d))\|=d_{\D^d}(z,w)\;\;(z,w\in\D^d).
\end{equation*}
This is a well-known fact. We recall that if $d\ge 3$, then $SA(B_{\delta_d})\subsetneq H(\D^d, \overline{\D})$. Therefore, it is reasonable for considering $S_2(B_{\Delta})$ rather than the Schur-Agler class in order to study Schwarz-Pick type inequalities.

(3) (Douglas-Paulsen class) For $r\in\mathbb{R}$ with $0<r<1$, we define a matrix of holomorphic functions on $\C^{\times}:=\C\backslash\{0\}$ by $a_r(z):=\mathrm{diag}(z, \frac{r}{z})$. Then, it defines the annulus $A(r)=\{z\in\mathbb{C}\;|\;r<|z|<1\}$. Let $\mathscr{S}(\mathcal{F}_{dp}(r))$ be the Douglas-Paulsen class \cite[Section 9.4]{AMY}. By \cite[Theorem 9.54]{AMY}, we have $\mathscr{S}(\mathcal{F}_{dp}(r))\subset S_2(B_{a_r})$. Therefore, every $f\in\mathscr{S}(\mathcal{F}_{dp}(r))$ must satisfy:
\begin{equation*}
\dd(f(z), f(w))\le\max\{\dd(z,w), \dd\left(\frac{r}{z}, \frac{r}{w}\right)\}\;\;(z,w\in A(r)).
\end{equation*}
In section 5, we will see that $d_{a_r}(z,w)=\max\{\dd(z,w), \dd\left(\frac{r}{z}, \frac{r}{w}\right)\}$ does not coincide with the M$\ddot{\mathrm{o}}$bius pseudo-distance $d_{A(r)}(z,w)$ in general.

\end{examples}


\section{Dilations of $2\times2$ matrices and geometry}

Here, we study possible relations between geometry with respect to $d_{\Delta}$ and analysis of $2\times 2$ matrices. In the rest of this paper, we assume that $\Delta$ is a matrix of holomorphic functions and $B_{\Delta}\subset\C^d$ is connected. 

 We begin by proving that inequality (\ref{eq:Schwarz-Pick}) is the best possible estimate for $S_2(B_{\Delta})$.
 \begin{prop}\label{prop1}
 For any pair $z, w\in B_{\Delta}$, there exists an $f\in S_2(B_{\Delta})$ such that 
 \begin{equation*}
 d_{\Delta}(z,w)=\dd(f(z), f(w)).
 \end{equation*}
 In particular, 
 \begin{equation}\label{eq:sup}
 d_{\Delta}(z, w)=\sup_{f\in S_2(B_{\Delta})}\dd(f(z), f(w)).
 \end{equation}
 Moreover, if $\delta$ is a matrix of polynomials, then 
 \begin{equation*}
 d_{\delta}(z,w)=\sup_{f\in S_2(B_{\delta})}\dd(f(z), f(w))=\sup_{f\in SA(B_{\delta})}\dd(f(z), f(w)).
 \end{equation*}
 \end{prop}
 \begin{proof}
 
 Let $z$ and $w$ be any points in $B_{\Delta}$. Define a $B(\C^r,\C^s)$-valued holomorphic function $g_w$ on $B_{\Delta}$ by 
 \begin{equation*}
 g_w(\zeta):=(I-\Delta(w)\Delta(w)^*)^{-\frac{1}{2}}(\Delta(\zeta)-\Delta(w))(I-\Delta(w)^*\Delta(\zeta))^{-1}(I-\Delta(w)^*\Delta(w))^{\frac{1}{2}} 
 \end{equation*}
 Since $\Delta(z)$ and $\Delta(w)$ act on a finite-dimensional Hilbert space, there exist unit vectors $\xi\in\C^r$ and $\eta\in\C^s$ such that
 \begin{equation*}
 |\langle g_w(z)\xi, \eta\rangle_{\C^s}|=d_{\Delta}(z,w).
 \end{equation*}
 We define a holomorphic function on $B_{\Delta}$ by
 \begin{equation*}
 f_{z,w}(\zeta):=\langle g_w(\zeta)\xi, \eta\rangle\;\;\;\;(\zeta\in B_{\Delta}).
 \end{equation*}
 Then, it is easy to see that $\dd(f_{z,w}(z), f_{z,w}(w))=d_{\Delta}(z, w)$. Thus, it remains to prove $f_{z,w}\in S_2(B_{\Delta})$. If $T$ is generic such that $\sigma(T)\subset B_{\Delta}$ and $\|\Delta(T)\|\le 1$, then the Neumann series expansion gives
 \begin{equation*}
 f_{z,w}(T)=(\eta^*\otimes I_2)g(T)(\xi\otimes I_2).
 \end{equation*}
 Here, $I_2$ is the $2\times 2$ identity matrix,
 \begin{align*}
 g(T)=&(I-\Delta(W)\Delta(W)^*)^{\frac{1}{2}}(\Delta(T)-\Delta(W))\\
 &\;\;\;\;\;\;\times(I-\Delta(W)^*\Delta(T))^{-1}(I-\Delta(W)^*\Delta(W))^{\frac{1}{2}}
 \end{align*}
 and the operator $\Delta(W)\in B(\C^r\otimes\C^2, \C^s\otimes\C^2)$ is given by $\Delta(W):=\Delta(w)\otimes I_2$.
 The same argument as in the proof of \cite[Theorem 2]{Har} with replacing $A$ and $B$ there with $\Delta(T) $ and $-\Delta(W)$ respectively will show that
 \begin{align*}
 I-g(T)^*g(T)=&(I-\Delta(W)^*\Delta(W))^{\frac{1}{2}}(I-\Delta(T)^*\Delta(W))^{-1}(I-\Delta(T)^*\Delta(T))\\
 &\;\;\times(I-\Delta(W)^*\Delta(T))^{-1}(I-\Delta(W)^*\Delta(W))^{\frac{1}{2}}\ge 0
 \end{align*}
 (see also \cite[equation (4.2)]{McT}). Hence, we have $\|g(T)\|\le 1$. Since $\|\xi\|=\|\eta\|= 1$, we conclude $\|f_{z,w}(T)\|\le 1$. Therefore, $f_{z,w}\in S_2(B_{\Delta})$.
 
 Let $\delta$ be a matrix of polynomials. By the transfer-function realization, we have $SA(B_{\delta})\subset S_2(B_{\delta})$. Therefore, it remains to see that
 \begin{equation*}
 \sup_{f\in S_2(B_{\delta})}\dd(f(z), f(w))\le\sup_{f\in SA(B_{\delta})}\dd(f(z), f(w)) .
 \end{equation*}
For any $f\in S_2(B_{\delta})$, Theorem \ref{Theorem1}(1)$\Rightarrow$(4) implies that there exists an $F\in SA(R_{sr})$ such that $f(z)=F(\delta(z))$ and $f(w)=F(\delta(w))$. The transfer-function realization shows again that the holomorphic function  $g(\zeta):= F(\delta(\zeta))$ falls in $SA(B_{\delta})$. Then, we have
\begin{equation*}
\dd(f(z), f(w))=\dd(F(\delta(z)), F(\delta(w)))=\dd(g(z), g(w)).
\end{equation*}
This implies
\begin{equation*}
 \sup_{f\in S_2(B_{\delta})}\dd(f(z), f(w))\le\sup_{f\in SA(B_{\delta})}\dd(f(z), f(w)).
 \end{equation*} \end{proof}
 
 \begin{remark}
 \upshape This result implies that $d_{\Delta}$ is a pseudo-distance on $B_{\Delta}$ because $\dd$ enjoys the triangle inequality. 
 \end{remark}
 
 If a domain $\Omega$ is bounded, then the M$\ddot{\mathrm{o}}$bius pseudo-distance $d_{\Omega}$ (hence the Carath\'{e}odory pseudo-distance $C_{\Omega}$ too) is a distance on $\Omega$. In general, $d_{\Delta}$ is not a distance of a bounded domain $B_{\Delta}$ (e.g. $E=\C$ and $\Delta=[z^2]$). However, if  $B_{\Delta}$ is bounded with respect to the generic matrices, then $d_{\Delta}$ becomes a distance on $B_{\Delta}$.
 

\begin{definition}
$B_{\Delta}$ is said to be 
{\bf 2-bounding} if 
\begin{equation*}
\sup\{\|T^r\|\;|\; 1\le r\le d,\;\text{$T$ is generic with $\sigma(T)\subset B_{\Delta}$ and $\|\Delta(T)\|\le1$}\}<\infty.
\end{equation*}
\end{definition}


\begin{prop}
If $B_{\Delta}$ is 2-bounding, then $d_{\Delta}$ is a distance on $B_{\Delta}$ i.e.,  $d_{\Delta}(z,w)=0$ implies $z=w$.
\end{prop}
\begin{proof}
The idea is same as a part of the proof of \cite[Theorem 4.18]{AM2015oka}. Suppose that $z$ and $w$ are distinct points in $B_{\Delta}$ with $d_{\Delta}(z,w)=0$. Then, we have $\Delta(z)=\Delta(w)$. 
Since $B_{\Delta}$ is 2-bounding, we may assume that the coordinate functions $\chi^r(z):=z^r$ $(r=1,\ldots, d)$ are in $S_2(B_{\Delta})$. Theorem \ref{Theorem1}(1)$\Rightarrow$(4) yields $F_r\in SA(R_{sr})$ with
\begin{equation*}
z^r=F_r(\Delta(z)),\;\; w^r=F_r(\Delta(w))\;\;(r=1,\ldots,d).
\end{equation*}
As $\Delta(z)=\Delta(w)$, we have $z=w$, a contradiction.
\end{proof}

We have already seen, in Examples \ref{examples:2.2}(2), the product property for the M$\ddot{\mathrm{o}}$bius pseudo-distance. A similar property holds for our pseudo-distance $d_{\Delta}$.


\begin{prop}\label{prop3}
If $\Delta_1$ and $\Delta_2$ are matrices of holomorphic functions, then we have
\begin{equation*}
d_{\Delta_1\oplus\Delta_2}((z^1, z^2), (w^1, w^2))=\max\{d_{\Delta_1}(z^1, w^1), d_{\Delta_2}(z^2, w^2)\}
\end{equation*}
for any $(z^1, z^2), (w^1, w^2)\in B_{\Delta_1}\times B_{\Delta_2}=B_{\Delta_1\oplus\Delta_2}$. Here,
\begin{equation*}
(\Delta_1\oplus\Delta_2)(z^1, z^2):=
\begin{bmatrix}
\Delta_1(z^1)&0\\
0&\Delta_2(z^2)
\end{bmatrix}.
\end{equation*}
\end{prop}
\begin{proof}
This is due to a direct calculation.
\end{proof}





Next, we will give a dilation type result. We give a sufficient condition on a diagonalizable commuting tuple $T$ acting on $\C^2$ for $B_{\Delta}$ to be a complete spectral domain for $T$. Of course, this is quite special case.  However, Drury \cite{Dru1983} observed that every commuting tuple of $2\times 2$ matrices can be approximated by diagonalizable commuting tuples. So, we hope that the present approach leads to studying dilation theory for general (non-diagonalizable) $2\times 2$ matrices.

We recall the definition of (complete) spectral domains.
\begin{definition}
A domain $\Omega\subset\C^d$ is said to be a ({\bf complete}) {\bf spectral domain} for a commuting $d$-tuple $T$ of operators if $\sigma(T)\subset\Omega$ and, (resp.\;for every positive integer $n$ and) every holomorphic (resp.\;$n\times n$-matrix-valued) function $f$ on $\Omega$,
\begin{equation*}
\|f(T)\|\le\sup_{z\in\Omega}\|f(z)\|
\end{equation*}
holds.
\end{definition}

Let us say that a commuting $d$-tuple $T$ of operators acting on $\C^2$ is {\bf diagonalizable} if there exist linearly independent vectors $v_1$ and $v_2$ in $\C^2$ and a pair of points $z_1$, $z_2\in\C^d$ {\bf (which may not be distinct)} such that
\begin{equation*}
T^rv_j=z_j^rv_j
\end{equation*}
for $1\le r\le d$ and $j=1,2$. (Note that, joint eigenvalues of a generic tuple of operators are distinct.)

In the proof of the next dilation theorem, we crucially use Agler's deep observations on relations between operator theory and complex geometry \cite{Agl1990-2}.


\begin{theorem}\label{theorem2}
Let $T$ be a diagonalizable commuting $d$-tuple of operators acting on $\C^2$ with $\sigma(T)=\{z,w\}\subset B_{\Delta}$. If $\|\Delta(T)\|\le 1$ and $d_{\Delta}(z,w)=d_{B_{\Delta}}(z,w)$, then $B_{\Delta}$ is a complete spectral domain for $T$. 
\end{theorem}
\begin{proof}
	Let $T$ be a diagonalizable $d$-tuple of operators acting on $\C^2$ with $\sigma(T)=\{z,w\}\subset B_{\Delta}$ and satisfying $\|\Delta(T)\|\le 1$. We denote the angle between the eigenspaces of $T$ by $\theta_T$ (its precise definition is in \cite[Section 8.2]{AMY}). By \cite[Lemma 8.6]{AMY} and \cite[equation (8.11)]{AMY}, we have
\begin{equation*}
\dd(f(z), f(w))\le \sin\theta_T
\end{equation*}
for all $f\in S_2(B_{\Delta})$. 
Therefore, Proposition \ref{prop1} implies 
\begin{equation*}
d_{\Delta}(z,w)\le\sin\theta_T.
\end{equation*}
Since $d_{\Delta}(z,w)=d_{B_{\Delta}}(z,w)$, it follows from \cite[Theorem 8.13]{AMY} that $B_{\Delta}$ is a spectral domain for $T$. It is known that a spectral domain for a diagonalizable tuple of operators acting on $\C^2$ is also a complete spectral domain for it; see \cite[Proposition 8.21]{AMY}.
\end{proof}


The next corollary is a geometric characterization of a domain, for which Drury's dilation type result \cite{Dru1983} holds. The corollary is Theorem \ref{theorem2} plus an additional equivalent condition.

\begin{cor}\label{cor3}
Let $B_{\Delta}$ be a domain in $\C^d$ associated with a matrix of holomorphic functions $\Delta$. Then the following conditions are equivalent:
\begin{enumerate}
\item $d_{\Delta}=d_{B_{\Delta}}$.
\item If a diagonalizable $d$-tuple $T$ of operators acting on $\C^2$ with $\sigma(T)\subset B_{\Delta}$ satisfies $\|\Delta(T)\|\le1$, then $B_{\Delta}$ is a complete spectral domain for $T$.
\item $S_2(B_{\Delta})=H(B_{\Delta}, \overline{\D})$.
\end{enumerate}
\end{cor}
\begin{proof}
(1)$\Rightarrow$(2) follows from Theorem \ref{theorem2}.

(2)$\Rightarrow$(3): Let $T$ be a generic $d$-tuple of operators with $\sigma(T)\subset B_{\Delta}$ and $\|\Delta(T)\|\le 1$. Then, $B_{\Delta}$ is a complete spectral domain for $T$. In particular, we have $\|f(T)\|\le 1$ for any $f\in H(B_{\Delta}, \overline{\D})$. Therefore, we conclude $S(B_{\Delta})=H(B_{\Delta}, \overline{\D})$.

(3)$\Rightarrow (1)$ follows from Proposition \ref{prop1}.
\end{proof}



Here is another corollary.

\begin{cor}\label{cor4}
Let $n_1,\ldots, n_d\in \mathbb{N}$ and let $\Delta_r$ be a matrix of holomorphic functions that defines a domain $B_{\Delta_r}\subset\C^{n_r}$ for each $r=1, \ldots, d$. Suppose that each $\Delta_r$ satisfies $d_{\Delta_r}=d_{B_{\Delta_r}}$. Then, whenever $T=(T^1,\ldots, T^d)$ is a diagonalizable $(n_1+\cdots +n_d)$-tuple of operators acting on $\C^2$ (each $T^r$ denotes an $n_r$-tuple of diagonalizable operators) such that $\sigma(T^r)\subset B_{\Delta_r}$ and $\|\Delta_r(T^r)\|\le 1$ $(r=1,\ldots, d)$, $B_{\Delta_1}\times\cdots\times B_{\Delta_d}=B_{\oplus_{r=1}^d\Delta_r}$ is a complete spectral domain for $T$.
\end{cor}
\begin{proof}
This follows from Proposition \ref{prop3}, Corollary \ref{cor3} and the direct product property of the M$\ddot{\mathrm{o}}$bius pseudo-distance (see Examples \ref{examples:2.2}(2)).
\end{proof}



\begin{examples}\upshape
(1) We generalize von Neumann's inequality for commuting $2\times 2$ row contractions \cite[Corollary 3.4]{HRS} and for commuting tuples of $2\times 2$ contractions \cite{Dru1983}. Let $T^1,\cdots, T^d$ be commuting row contractions acting on $\C^2$. 
Moreover, we assume that the entries of $T^r=(T_1^r,\ldots, T_{n_r}^r)$ commute with the entries of $T^s=(T_1^s,\ldots, T_{n_s}^s)$ for any $r\ne s\in\{1,\ldots, d\}$.
We denote the open unit ball in $\C^d$ by $\mathbb{B}_d$. We will prove that the closed polyball $\overline{\mathbb{B}_{n_1}}\times\cdots\times \overline{\mathbb{B}_{n_d}}\subset\C^{n_1+\cdots+n_d}$ is a complete spectral set for $T:=(T^1,\ldots, T^d)$. Here, a compact set $K$ is said to be a {\bf complete spectral set} for a commuting tuple of matrices $T$ if $\sigma(T)\subset K$ and, for every positive integer $n$ and every $n\times n$-matrix-valued bounded rational functions on $K$,
\begin{equation*}
\|r(T)\|\le\sup_{z\in K}\|r(z)\|
\end{equation*}
holds. 

Since each $T^r$ is a row contraction, we have $\sigma(T^r)\subset\overline{\mathbb{B}_{n_r}}$ ($r=1,\ldots, d$) (see e.g., \cite{HRS}). Therefore, the spectral mapping property for the Taylor spectrum \cite{Cur} implies that $\sigma(T)\subset \overline{\mathbb{B}_{n_1}}\times\cdots\times\overline{\mathbb{B}_{n_d}}$. For each $r\in\{1,\ldots, d\}$, set $\Delta_r(z):=[z^1\cdots z^{n_r}]$. We note that $B_{\Delta_r}=\mathbb{B}_{n_r}$ and $d_{\Delta_r}=d_{\mathbb{B}_{n_r}}$ (see \cite{DW, Har}). Since every commuting tuple of operators acting on $\C^2$ can be approximated by diagonalizable tuples (see \cite{Dru1983}), we can choose a suitable sequence $\{\rho_n^r\}_{n=1}^{\infty}\subset (0,1)$ and diagonalizable tuples $\{T_n^r\}_{n=1}^{\infty}$ so that $\sigma(\rho_n^rT_n^r)\subset B_{\Delta_r}=\mathbb{B}_{n_r}$, $\|\Delta_r(\rho_n^rT_n^r)\|\le 1$ and $\rho_n^rT_n^r$ converges to $T^r$ as $n\rightarrow\infty$ for each $r$. 
By Corollary \ref{cor4}, $\mathbb{B}_{n_1}\times\cdots\times\mathbb{B}_{n_d}$ is a complete spectral domain for $(\rho_n^1T_n^1,\ldots, \rho_n^dT_n^d)$. Therefore, we can conclude that $\overline{\mathbb{B}_{n_1}}\times\cdots\times \overline{\mathbb{B}_{n_d}}$ is a complete spectral set for $T$.

(2) We show that $d_{a_r}$ does not coincide with $d_{A(r)}$. Recall that a $2\times 2$ matrix-valued function $a_r$ on $\C^{\times}:=\C\backslash\{0\}$ with $0<r<1$ is defined by $a_r(z):=\mathrm{diag}(z, \frac{r}{z})$. We note that $z\mapsto \frac{1}{\sqrt{r}}z$ is a biholomorphic map from $A(r)\;(=B_{a_r})$ onto $C(r):=\{z\in\C\;|\;\sqrt{r}<|z|<\frac{1}{\sqrt{r}}\}$ (n.b., it is well known that the annulus is essentially determined by the ratio of its inner radius and outer radius \cite[Theorem 14.22]{Rud}). It is easy to check that $d_{a_r}(\sqrt{r}, -\sqrt{r})=\displaystyle\frac{2\sqrt{r}}{1+r}$. Since the M$\ddot{\mathrm{o}}$bius distance is invariant under biholomorphic maps, we have $d_{A(r)}(\sqrt{r}, -\sqrt{r})=d_{C(r)}(1, -1)$. By \cite{Sim}, we can see that
\begin{equation*}
d_{A(r)}(\sqrt{r}, -\sqrt{r})=4\sqrt{r}\frac{\prod_{n=1}^{\infty}(1+r^{2n})^4}{\prod_{n=1}^{\infty}(1+r^{2n-1})^4}.
\end{equation*}
It is easy to see that $d_{a_r}(\sqrt{r}, -\sqrt{r})<d_{A(r)}(\sqrt{r}, -\sqrt{r})$ for any $0<r<\sqrt[3]{2}-1$. By Theorem \ref{cor3}, there exists a diagonalizable operator $T$ acting on $\C^2$ satisfying $\sigma(T)\subset A(r)$ and $r\le\|T\|\le 1$ such that $A(r)$ is not a spectral domain for $T$. In contrast, Pal and Tomar \cite{PT} gave such an example that is not diagonalizable.

\end{examples}


\begin{remark}\upshape
We have known that a noncommutative analog of Schwarz lemma holds for the noncommutative Schur-Agler class \cite{Koj}. Thus, it is natural to consider a noncommutative analog of Schwarz-Pick lemma. Unfortunately, our present approach collapses in the noncommutative setting because the automorphisms of a noncommutative Cartan domain of type I are not transitive \cite{McT}.
\end{remark}




\section*{Acknowledgment}

The author acknowledges his supervisor Professor Yoshimichi Ueda for his encouragements. The author also acknowledges Professor John Edward McCarthy for his some comments, especially, concerning Remark \ref{remark}.

\end{document}